\documentclass[11pt]{amsart}
\usepackage[parfill]{parskip}    % Activate to begin paragraphs with an empty line rather than an indent
\usepackage{amssymb, amsthm, amsmath, mathrsfs}
\usepackage{color}
\usepackage{soul}
\usepackage{url}
\usepackage{pinlabel}
\usepackage{caption}
\usepackage{subcaption}

\usepackage{mathtools}  

\usepackage{tikz}  
\usetikzlibrary{decorations.pathreplacing}
\makeatletter
\def\underbrace#1{%
  \@ifnextchar_{\tikz@@underbrace{#1}}{\tikz@@underbrace{#1}_{}}}
\def\tikz@@underbrace#1_#2{%
  \tikz[baseline=(a.base)] {\node[inner sep=2] (a) {\(#1\)};
  \draw[line cap=round,decorate,decoration={brace,amplitude=4pt}]
    (a.south east) -- node[pos=0.5,below,inner sep=7pt] {\(\scriptstyle #2\)} (a.south west);}}

\def\overbrace#1{%
  \@ifnextchar^{\tikz@@overbrace{#1}}{\tikz@@overbrace{#1}^{}}}
\def\tikz@@overbrace#1^#2{%
  \tikz[baseline=(a.base)] {\node[inner sep=2] (a) {\(#1\)};
  \draw[line cap=round,decorate,decoration={brace,amplitude=4pt}]
    (a.north west) -- node[pos=0.5,above,inner sep=7pt] {\(\scriptstyle #2\)} (a.north east);}}
\makeatother    
\theoremstyle{plain}
\newtheorem{theorem}{Theorem}

\newtheorem{claim}[theorem]{Claim}

\newtheorem{conjecture}[theorem]{Conjecture}

\newtheorem{lemma}[theorem]{Lemma}
\newtheorem{proposition}[theorem]{Proposition}

\newtheorem{question}[theorem]{Question}

\newtheorem*{conjectureA}{Conjecture~\ref{conj:essconway}}
\newtheorem*{conjectureB}{Conjecture~\ref{conj:nstringprime}}

\theoremstyle{definition}

\newtheorem{remark}[theorem]{Remark}

\newcommand{\Z}{\mathbb{Z}}

\newcommand{\bji}{b_j^{(i)}}
\newcommand{\aji}{a_j^{(i)}}

\newcommand{\bdry}{\partial}
\newcommand{\nbhd}{\mathcal{N}}
\DeclareMathOperator{\rk}{rk}

\title{Montesinos knots, Hopf plumbings, and L-space surgeries}
\author{Kenneth L.\ Baker and Allison H.\ Moore}
\address{Department of Mathematics\\University of Miami\\\newline Coral Gables, FL 33146 \\ USA}
\email{k.baker@math.miami.edu}
\address{Department of Mathematics\\Rice University\\\newline Houston, TX 77005 \\ USA}
\email{allison.h.moore@rice.edu}

\date{}                                           % Activate to display a given date or no date

\begin{document}

\begin{abstract}
Using Hirasawa-Murasugi's classification of fibered Montesinos knots we classify the L-space Montesinos knots, providing further evidence towards a conjecture of Lidman-Moore that L-space knots have no essential Conway spheres.  In the process, we classify the fibered Montesinos knots whose open books support the tight contact structure on $S^3$.  We also construct L-space knots with arbitrarily large tunnel number and discuss the question of whether L-space knots admit essential tangle decompositions in the context of satellite operations and tunnel number.
\end{abstract}

\maketitle

%%%%%%%%%%%%%%%%%%%%%%%
%%%%%%%  INTRO  %%%%%%%%%%%
%%%%%%%%%%%%%%%%%%%%%%%

\section{Introduction}
From the algebraic viewpoint of Heegaard Floer homology, {\em L-spaces} are the ``simple'' $3$--manifolds.  They are the rational homology spheres with rank as small as possible, i.e.\ the manifolds $Y$ such that $\rk \widehat{HF}(Y) = | H_1(Y;\Z) |$.  These include $S^3$, the lens spaces (except $S^1 \times S^2$), the other elliptic manifolds, many Seifert fibered spaces, as well as many hyperbolic manifolds \cite{OS-lens}.   
One way to construct examples of L-spaces is through ``bootstrapping'' a known L-space surgery on a knot.  
It follows from the Heegaard Floer surgery exact triangle that if $r>0$ surgery on a knot $K$ in $S^3$ is an L-space, then for every number $s\geq r$ the result of $s$--surgery on $K$ is also an L-space \cite{OS-lens}.  Capturing this, a knot $K$ in $S^3$ admitting a {\em positive} Dehn surgery to an L-space is known as an {\em L-space knot}.   

So which knots are L-space knots? This question would be answered with a geometric characterization.   Already established are the fiberedness \cite{ni} and support of the tight contact structure \cite[Corollary~1.4 with Proposition~2.1]{hedden} for L-space knots.   Using these properties, the structure of the Alexander polynomial of an L-space knot  \cite{OS-lens}, and the determinant-genus inequality for L-space knots \cite[Lemma~5]{lidmanmoore} we further probe a conjecture about geometric decompositions of L-space knots.
 
\begin{conjectureA}[Lidman-Moore \cite{lidmanmoore}]%\label{conj:essconway}
An L-space knot has no essential Conway sphere.
\end{conjectureA}

To do so, we first extend the results and techniques of \cite{lidmanmoore} to obtain a classification of the L-space knots among the Montesinos knots.
\begin{theorem}\label{thm:main}
Among the Montesinos knots, the only  L-space knots are the pretzel knots  $P(-2,3,2n+1)$ for $n\geq0$ and the torus knots $T(2,2n+1)$ for $n\geq0$.
\end{theorem}
 It turns out that generalizing from pretzel knots to Montesinos knots yields no new L-space knots than the ones already obtained in \cite{lidmanmoore} and thus continues to support Conjecture~\ref{conj:essconway}.  
 
 As a byproduct of our proof, we obtain a classification of fibered Montesinos knots that support the tight contact structure.  The statement of the following theorem uses the conventions of Hirasawa-Murasugi for the notation of Montesinos links \cite{HM-monty} which we review at the beginning of section~\ref{sec:monty}.

\begin{theorem}\label{thm:fibered}
A fibered Montesinos knot that supports the tight contact structure is isotopic to either
\begin{itemize}
\item $M(\tfrac{-d_1}{2d_1+1},\dots, \tfrac{-d_r}{2d_r+1} \big|1 )$ for some set of positive integers $d_1, \dots, d_r$ such that $d_1 + \dots + d_r$ is even, or
\item $M( \tfrac{-m_1}{m_1+1}, \dots, \tfrac{-m_r}{m_r+1} \big| 2 )$ for some odd integer $m_1 \geq 1$ and even integers $m_2, \dots, m_r \geq 2$.
\end{itemize}
Moreover, the knot with its fiber may be obtained from the disk by a sequence of Hopf plumbings.
\end{theorem}

These two families of fibered Montesinos knots are illustrated with their fibers in Figure~\ref{fig:posfiberedmonty}.
\begin{figure}[h]
	\centering
                	\includegraphics[height=1in]{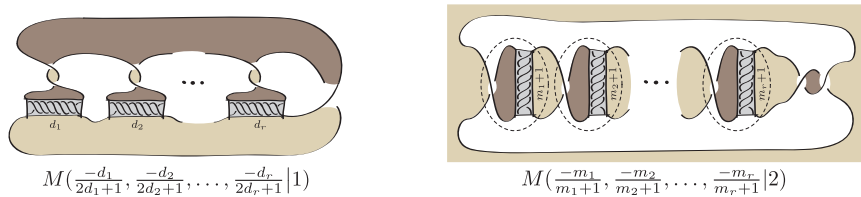} 
		\caption{The two families of fibered Montesinos knots that support the tight contact structure.}
    	\label{fig:posfiberedmonty}
\end{figure}

These two theorems will be proven in section~\ref{sec:mainproofs} with a discussion of the general strategy in section~\ref{sec:strategy}.    Lemma~\ref{lem:twobridge} recalls the corresponding result for two-bridge knots.  Then our arguments split according to Hirasawa-Murasugi's partition of Montesinos knots into odd and even types.

\begin{proof}[Proof of Theorem~\ref{thm:fibered}]
Lemma~\ref{lem:twobridge} handles two-bridge knots.
Proposition~\ref{prop:correctform} produces the first family for odd type Montesinos knots;  Proposition~\ref{prop:eventight} produces the second family for even type Montesinos knots.
\end{proof}

\begin{proof}[Proof of Theorem~\ref{thm:main}]
Lemma~\ref{lem:twobridge} shows the L-space two-bridge knots are the $T(2,2n+1)$ torus knots for $n\geq0$. We then restrict attention to Montesinos knots of length at least $3$ (since those of shorter length are two-bridge).
Then, among these knots, Proposition~\ref{prop:nooddlspace} shows there are no L-space knots of odd type and Proposition~\ref{prop:evenLspace} shows that those of even type are the pretzel knots $P(-2, 3, 2n+1)$ for integers $n\geq0$.
\end{proof}

Thereafter, in section~\ref{sec:esstangles}, we generalize Conjecture~\ref{conj:essconway}.
\begin{conjectureB}
L-space knots have no essential tangle decomposition.
\end{conjectureB}
We examine this conjecture in the contexts of satellite L-space knots and tunnel numbers of L-space knots and pose a few questions.

 %%%%%%%%%%%%%%%%%%%%%%%%%%%%%

\section{Preliminaries}
\subsection{Open books}

We refer the reader to the lecture notes \cite{etnyre-openbooks} of Etnyre for a useful survey on the basics of open books and contact structures.  Nevertheless, let us remind the reader of a few items.

An {\em open book} for an oriented $3$--manifold $Y$ is a link $L$ with a fibration of its complement  $\phi \colon Y-L \to S^1$ such that a fiber $F = \phi^{-1}(0)$ is a Seifert surface for the link, $\bdry F = L$.   Since the binding of an open book is an oriented fibered link, we may simply speak of the fibered link (up to orientation reversal of all the components) since a fiber and hence the fibration will be understood. Each open book for a $3$--manifold induces a contact structure on that manifold \cite{thurstonwinkelnkemper}; more precisely, it {\em supports} a unique contact structure \cite{giroux}. Contact structures on $3$--manifolds can either be {\em tight} or {\em overtwisted}.  On $S^3$ there is a unique tight contact structure.

The positive and negative Hopf bands, $H^+$ and $H^-$, are shown in Figure~\ref{fig:hopfbands} left and center.   
Let us say a Seifert surface {\em contains} a positive or negative Hopf band if one may be deplumbed from the surface.

Following \cite{yamamoto}, an essential simple closed curve in a Seifert surface is a {\em twisting loop} if it bounds a disk in the manifold with the same framing as the Seifert surface. For example, the connected sum of a positive and negative Hopf band contains a twisting loop, shown in Figure~\ref{fig:hopfbands} right.
\begin{figure}[h]
\centering
\includegraphics[height=1in]{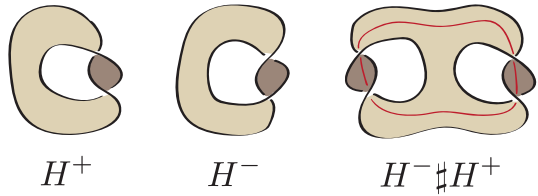}
\caption{The positive and negative Hopf bands, and their connected sum contains a twisting loop as shown.}
\label{fig:hopfbands}
\end{figure}

\begin{lemma}\label{lem:OTmarkers} 
%Let $(Y, F)$ denote an open book.
Let $F$ be a fiber of an open book.
\begin{itemize} 
	\item If $F$ contains a positive Hopf band, then its open book supports the same contact structure as the open book obtained by deplumbing that positive Hopf band, \cite{giroux}.
	\item If $F$ contains a negative Hopf band, then its open book supports an overtwisted contact structure, e.g. \cite{etnyre-openbooks}.
	\item If $F$ contains twisting loop, then its open book supports an overtwisted contact structure,  \cite[Theorem 1.1]{yamamoto}.  \qed
\end{itemize}
\end{lemma}

\subsection{A basic sequence of deplumbings}

\begin{figure}[h]
	\centering
                	\includegraphics[scale=.4]{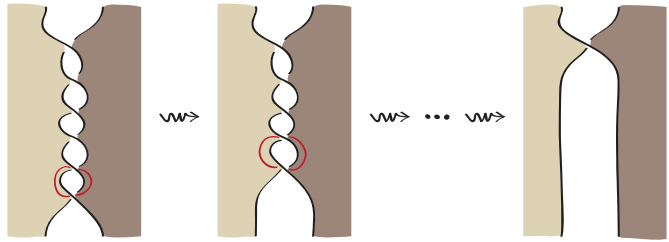}
	\caption{When a surface laterally encounters a twist region, there is a sequence of ``obvious deplumbings''.  Shown on the left is a vertical twist region of $6$ negative half-twists from which $5$ negative Hopf bands are successively deplumbed.}
    	\label{fig:Sketch0}
\end{figure}

\begin{lemma}\label{lem:obvioushopf}
Let $F$ be a Seifert surface for a link $L$ with a sequence of $n$ half-twists in $L$ as shown in the left side of Figure~\ref{fig:Sketch0} for $n=-6$.  If $n \geq 2$, then $F$ contains $n-1$ positive Hopf bands which may be successively deplumbed, leaving a single positive half twist. lf $n \leq -2$, then $F$ contains $|n|-1$ negative Hopf bands which may be successively deplumbed, leaving a single negative half twist.
\end{lemma}
\begin{proof}
The case of $n=-6$ is shown in Figure~\ref{fig:Sketch0} and makes an inductive proof clear.  Mirror the figure for positive $n$.
\end{proof}

%%%%%%%%%%%%%%%%%%%%%%%%%%%%%%%%%%%%%

\subsection{Montesinos knots and links}\label{sec:monty}
As stated in the introduction, we follow the conventions of Hirasawa-Murasugi for our notation \cite{HM-monty}.
A Montesinos link $K$ is denoted   
\[ 
	K=M\left( \frac{\beta_1}{\alpha_1}, \frac{\beta_2}{\alpha_2,} \dots, \frac{\beta_r}{\alpha_r} , \,\vert\, e \right) 
\]
where $\alpha_i >1$, $|\beta_i| < \alpha_i$, and $\gcd(\alpha_i, \beta_i) =1$.  The number $r$ is the {\em length} of the Montesinos link.
As an illustration, Figure~\ref{fig:montyexample} shows the length $3$ Montesinos knot $M(\tfrac{3}{4}, -\tfrac{2}{5}, \tfrac{1}{3} \vert 3)$.
\begin{figure}[h]
	\centering
                	\includegraphics[width=.9\textwidth]{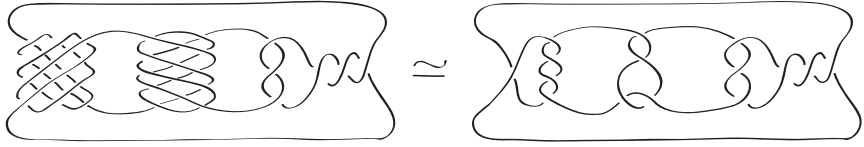}
	\caption{Two isotopic presentations of the Montesinos knot $M(\tfrac{3}{4}, -\tfrac{2}{5}, \tfrac{1}{3} \vert 3)$.}
    	\label{fig:montyexample}
\end{figure}

If the Montesinos link $K$ above is a {\em knot} then at most one of $\alpha_1, \dots, \alpha_r$ is even. (Note that this is not a sufficient condition for being a knot.)   By an isotopy of $K$, one may cyclically permute indices so that $\alpha_2, \dots, \alpha_r$ are all odd.  With this setup, Hirasawa-Murasugi then partition Montesinos knots into {\em odd types} and {\em even types} according to whether or not $\alpha_1$ is odd.  Hirasawa-Murasugi's Theorems~3.1 and 3.2 describe both the genera and fiberedness of Montesinos knots for odd types and even types, respectively \cite{HM-monty}.

{\bf Notation:}  In this article we will write, for example, O(II-3-ii) to refer to condition (II)(3)(ii) of Theorem~3.1 and E(III-i) to refer to condition (III)(i) of Theorem~3.2 of Hirasawa-Murasugi \cite{HM-monty}.

Hirasawa-Murasugi use special forms of continued fractions for the terms $\frac{\beta_i}{\alpha_i}$ in the notation of a Montesinos knot to describe minimal genus Seifert surfaces.  They define the continued fraction $S = [x_1, x_2, \dots, x_m]$ for a rational number $\frac{\beta}{\alpha}$ with $-\alpha<\beta<\alpha$  as the expression
\[ \frac{\beta}{\alpha} = \cfrac{1}{x_1 - \cfrac{1}{x_2-\cfrac{1}{\ddots - \cfrac{1}{x_m}}}}\]
where every coefficient $x_i$ is non-zero.   The continued fraction $S$ is said to be {\em even} if each coefficient $x_i$ is even.  The continued fraction $S$ is said to be {\em strict} if for each odd $j$ both (a) $x_j$ is even and (b) if $x_j = \pm2$ then $x_{j+1}$ has the opposite sign.

We will also need the determinant of a Montesinos knot.  This can be calculated as the order of the first homology of the double branched cover of the knot.   Using \cite{saveliev} as a reference and making adjustments for differences in notation, one obtains the following.
\begin{lemma}\label{lem:determinant}
The determinant of the Montesinos knot
$K=M\left( \frac{\beta_1}{\alpha_1}, \frac{\beta_2}{\alpha_2,} \dots, \frac{\beta_r}{\alpha_r} \,\vert\, e \right)$
is 
\[ \det(K) = \left| \prod_{i=1}^r \alpha_i \left( e + \sum_{i=1}^r \frac{\beta_i}{\alpha_i} \right ) \right|.\]
\end{lemma}

\section{Proofs of Theorems \ref{thm:main} and \ref{thm:fibered}}\label{sec:mainproofs}

\subsection{Strategy of proof of Theorem~\ref{thm:main}}\label{sec:strategy}
Recall the fundamental properties of L-space knots noted in the introduction:
\begin{itemize}
\item The non-zero coefficients of the Alexander polynomial of an L-space knot are $\pm1$ and alternate in sign \cite{OS-lens}.
\item An L-space knot $K$ of genus $g(K)$ satisfies the {\em determinant-genus inequality}: $\det(K) \leq 2g(k)+1$,  \cite[Lemma~5]{lidmanmoore}.
\item An L-space knot is fibered \cite{ni}.  
\item An  L-space knot supports the tight contact structure; that is, it is a fibered strongly quasipositive knot  \cite[Corollary~1.4 with Proposition~2.1]{hedden}.
\end{itemize}
These properties suggest a general strategy for identifying L-space knots among some collection of knots. 
 Briefly, the strategy is: (1) select the knots in the collection which are fibered and (2) support the tight contact structure, (3) cull the knots which do not satisfy the determinant-genus inequality, and (4) discard those whose Alexander polynomials do not have the correct form.  Ideally, at this stage the remaining knots may be recognized as L-space knots; but if not, (5) compute the knot Floer homology of the knots or the Heegaard Floer homology of large surgeries on the knots.

Following this strategy for the proof of Theorem~\ref{thm:main}, we (1) appeal to Hirasawa-Murasugi's classification of fibered Montesinos knots \cite{HM-monty} and then (2) cull those that either admit a Stallings twist or may have a negative Hopf band deplumbed (since by Lemma~\ref{lem:OTmarkers} these indicate that a fibered knot supports an overtwisted contact structure). At this point we certify that the remaining knots indeed do support the tight contact structure by showing they can all be obtained by successive plumbings of positive Hopf bands.   This completes stage (2) of the strategy producing two manageable families of Montesinos knots. (Truth be told, one of these two families are identifiable as pretzel knots, so we simply invoke the results of \cite{lidmanmoore}.  Nevertheless, one could continue with the strategy instead.)  Stage (3) then follows in a more-or-less straightforward calculation from formulae for the determinant of Montesinos knots which significantly reduces the set of Montesinos knots to be considered. For stage (4), computations of the Alexander polynomials are then sufficient to rule out all the knots that are not already known to be L-space knots.  Fortunately, stage (5) is unnecessary.   

Observe that stage (2) of the strategy produces the collection of fibered Montesinos knots that support the tight contact structure, Theorem~\ref{thm:fibered}. As a consequence of \cite[Corollary~1.4 with Proposition~2.1]{hedden}, this collection also describes the set of fibered Montesinos knots which are strongly quasipositive.

\begin{question}
Which non-fibered Montesinos knots are strongly quasipositive?
\end{question}

\subsection{Two-bridge knots and links}
Montesinos links of length $r=1$ or $2$ are two-bridge links.

\begin{lemma}\label{lem:twobridge}
The fibered two-bridge links that support the tight contact structure are the torus links $T(2,N)$ for integers $N\geq1$.

The L-space two-bridge knots are the torus knots $T(2,2n+1)$ for integers $n\geq0$.
\end{lemma}

\begin{proof}
We offer a quick sketch.
Fibered two-bridge knots and links are well-known to be obtained as a linear chain of plumbings of positive and negative Hopf bands (e.g.\ \cite{GKpA}); they have corresponding continued fraction expansions where each coefficient is $\pm2$.  Among these, only those built from positive Hopf bands support the tight contact structure on $S^3$ (Lemma~\ref{lem:OTmarkers}), and these happen to be the torus links $T(2,N)$ for integers $N\geq 1$.  When $N=2n+1$ is odd, these torus links are knots and $(2N-1)$--surgery is a lens space.  Thus the fibered two-bridge knots that support the tight contact structure are the L-space two-bridge knots.
\end{proof}

Alternatively, since two-bridge knots are alternating \cite{goodrick}, Theorem 1.5 of \cite{OS-lens} implies that the only two-bridge knots admitting L-space surgeries are those which are isotopic to a torus knot $T(2, 2n + 1)$ for some integer $n$.

We will henceforth assume that $r\geq3$.

%%%%%%%%%%%%%%%%%%%%%%%%%%%%%%%%%%%
%%%%%%%%%%%%%%%%%%%  ODD  %%%%%%%%%%%%
%%%%%%%%%%%%%%%%%%%%%%%%%%%%%%%%%%%

\subsection{Odd fibered Montesinos knots}

\begin{proposition}
\label{prop:correctform}
Let $K$ be an odd fibered Montesinos knot supporting the tight contact structure. Then for some set of positive integers $d_1, \dots, d_r$ such that $d_1 + \dots + d_r$ is even,
\[
	K = M( \tfrac{-d_1}{2d_1+1},\dots, \tfrac{-d_r}{2d_r+1} \vert1 ).
\]
Moreover, the fiber of $K$, shown on the left side of Figure~\ref{fig:posfiberedmonty}, is a positive Hopf plumbing.
\end{proposition}

\begin{proof}
Let $K$ be a fibered Montesinos knot of odd type.
By Theorem 3.1 of Hirasawa-Murasugi \cite{HM-monty}, we may assume that $e\neq 0$ and each $\beta_i/\alpha_i$ has a strict continued fraction expansion
\[
	S_i = [ 2a_1^{(i)}, b_1^{(i)}, \dots, 2a_{q_i}^{(i)}, b_{q_i}^{(i)}]
\]
where $|a_j^{(i)}| = 1$ or $2$ for all $i,j$ and $a_1^{(i)}$ has sign opposite $e$.
Note that strictness of $S_i$ implies that if $|a_j^{(i)}|=1$, then $a_j^{(i)}$ and $\bji$ differ in sign. The fiber $F$ of $K$ appears as in Figure~10 of \cite{HM-monty}. 

(For notational purposes, one may care to define $b_0^{(i)}$ as $-e$.  Indeed, for each $i$, there is an isotopy of $F$ so that the $e$ twist region plays the role of a $b_0^{(i)}$ twist region.)

In accordance with Lemma~\ref{lem:OTmarkers}, we assume that $F$ has no negative Hopf bands. 
Immediately Lemma~\ref{lem:obvioushopf} gives 
\begin{itemize}
\item[($\ast$)] either $e = +1$ or $e<0$ and either $\bji=-1$ or $\bji >0$.
\end{itemize}

\begin{claim}\label{claim:aij1}
For each $i$, $|\aji|=1$ implies $\aji=-1$, $\bji>0$, and either $b_{j-1}^{(i)} = -1$ or $e=+1$ if $j-1=0$.
\end{claim}

\begin{proof}
First assume $\aji =+1$.  Then $\bji<0$ by the strictness of $S_i$.
If $j=1$, then O(II-2-i) implies $e<0$.  An isotopy of the fiber moves one of the $e$ half-twists into the position shown in Figure~\ref{fig:baaB} where a negative Hopf band is found. 
If $j>1$, then O(II-3-i) implies $b_{j-1}^{(i)}>0$.  Thus near the $\aji$ twists, the surface locally appears as in Figure~\ref{fig:baaB} where a negative Hopf band is evident.
Thus $\aji=-1$.  The strictness of $S_i$ then implies $\bji>0$.  Condition O(II-3-i) with ($\ast$) then implies $b_{j-1}^{(i)} = -1$.
\end{proof}

\begin{figure}
\centering
\includegraphics[scale=.4]{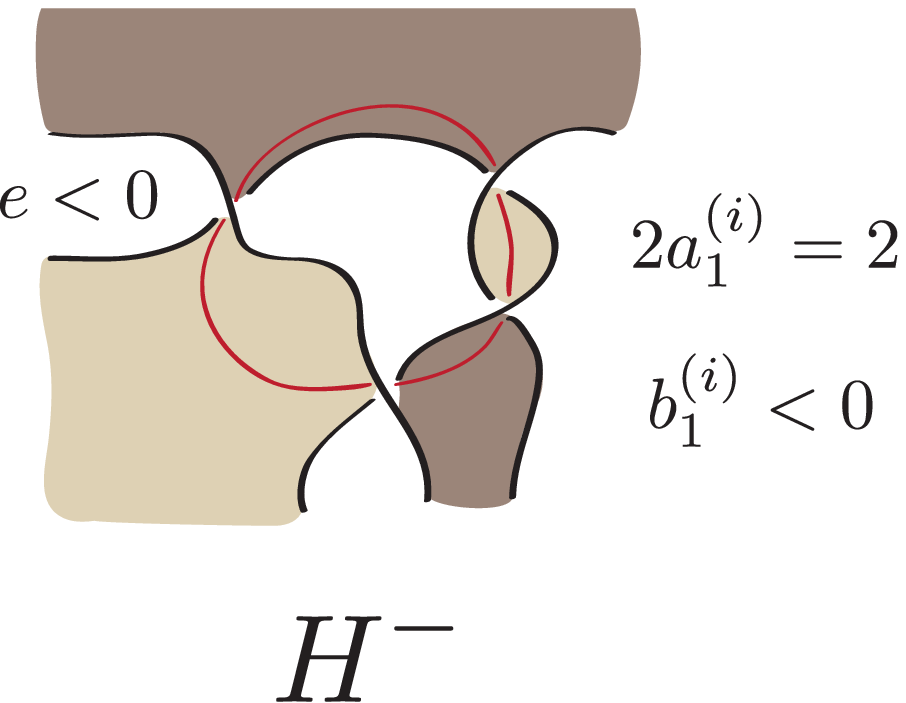} \quad \includegraphics[scale=.4]{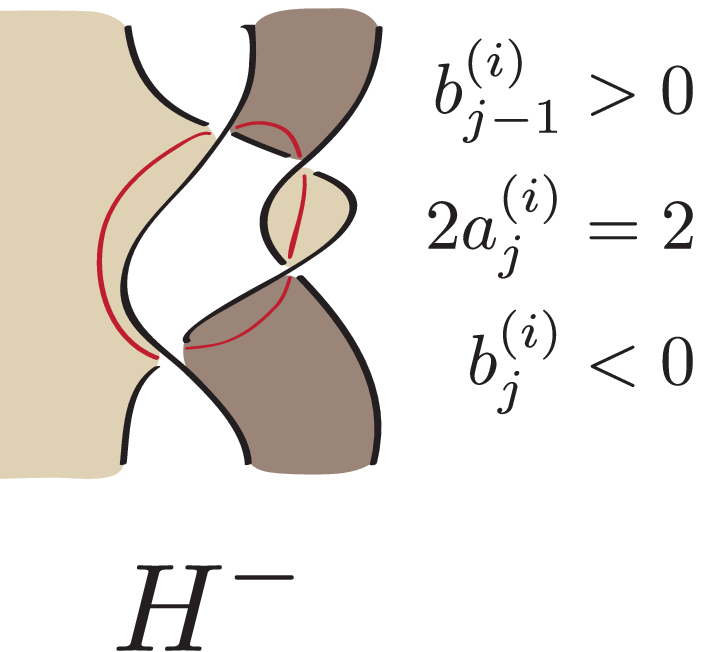}
\caption{When $a_j^{(i)}=1$ a negative Hopf band can be found. Left: $j=1$ after an isotopy of the surface.  Right: $j>1$.} 
\label{fig:baaB}
\end{figure}

\begin{figure}
\centering
\includegraphics[scale=.4]{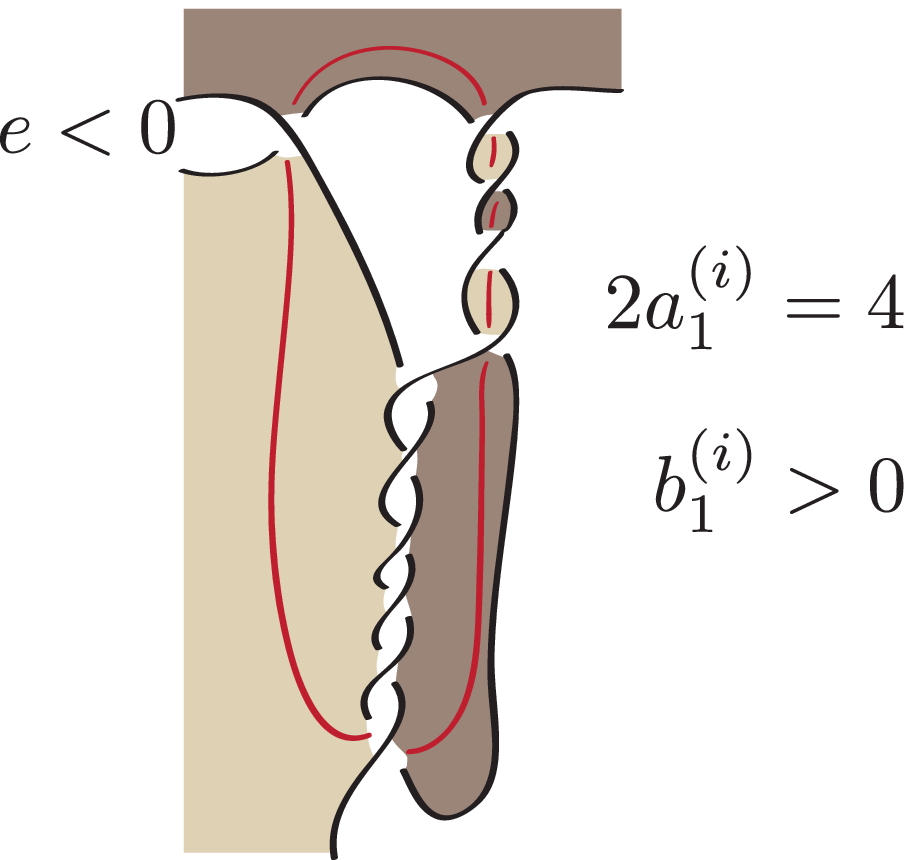} \quad\quad\quad \includegraphics[scale=.4]{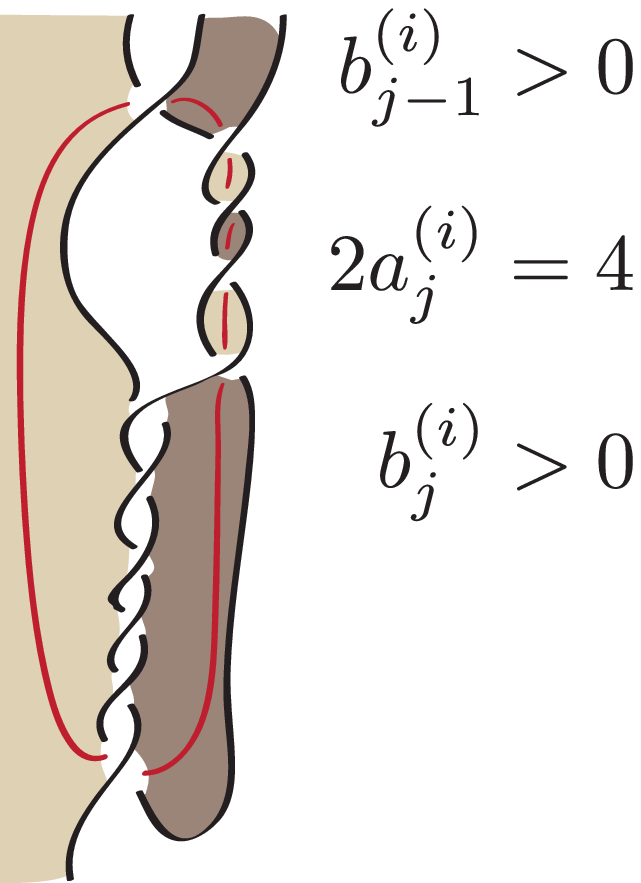}
\caption{When $a_j^{(i)}=2$ and $j$ is the final index, a negative Hopf band can be found. Left: $j=1$ after an isotopy of the surface.  Right: $j>1$.} 
\label{fig:BAAAAB}
\end{figure}

\begin{claim}\label{claim:aij2}
For each $i$, $|\aji|=2$ implies $\aji=-2$, $\bji=-1$, either $b_{j-1}^{(i)} = -1$ or $e=+1$ if $j-1=0$.
\end{claim}
\begin{proof}
First observe that if $\aji=-2$, the remaining conditions follow from O(II-3-ii), O(II-2-ii), and ($\ast$).

Let $j$ be the last index for which $\aji = +2$.  
If $j>1$, then O(II-3-ii) implies $\bji>0$ and $b_{j-1}^(i) >0$. If $j=1$, then O(II-2-ii) implies $\bji>0$ and $e<0$.   In either case, Claim~\ref{claim:aij1} and the first sentence of this proof then imply that $j$ must be the final index of $S_i$.
When $j=1$,   an isotopy of the fiber moves one of the $e$ half-twists into the position shown in Figure~\ref{fig:BAAAAB} where a negative Hopf band is found.
When $j>1$,  the surface locally appears near the $\aji$ twists as in Figure~\ref{fig:BAAAAB} where a negative Hopf band is evident.
\end{proof}

Let $(-4,-1)^{[n]}$ denote the sequence $-4,-1,-4,-1,\dots, -4,-1$ of length $2n$.
Together Claims~\ref{claim:aij1} and \ref{claim:aij2} then imply that $e=+1$ and for each strict continued fraction $S_i$, either $S_i=[(-4,-1)^{[n]} ]$ for $n\geq1$ or $S_i=[(-4,-1)^{[n]},-2,d_i]$ for $n\geq 0$ and $d_i>0$. Observing the equivalence $[\dots,-4,-1] = [\dots,-2,+1]$, we may assume the latter of these holds for $S_i$.  Figure~\ref{fig:212b_to_b1} then shows how to transform $S_i$ into the strict continued fraction $[-2,d_i]$ for some integer $d_i>0$.  Note that $[-2,d_i] = \frac{-d_i}{2d_i+1}$.  Thus any odd type fibered Montesinos knot without negative Hopf bands may be expressed as
\[K = M( \tfrac{-d_1}{2d_1+1},\dots, \tfrac{-d_r}{2d_r+1} \vert1 )\]
for some set of positive integers $d_1, \dots, d_r$.  

Figure~\ref{fig:posfiberedmonty} (left) illustrates Montesinos links of the form $M( \tfrac{-d_1}{2d_1+1},\dots, \tfrac{-d_r}{2d_r+1} \vert1 )$.  Performing a particular crossing change in each factor transforms this link to $M(-d_1, \dots, -d_r \vert 1)$ which may be recognized as the torus link $T(d_1+ \dots +d_r -1, 2)$.  Thus this link, and hence $K$ above, is a knot precisely when $d_1 + \dots + d_r$ is even.

To show that such a fibered knot actually supports the tight contact structure, we show that positive Hopf bands can be successively deplumbed from a fiber until a single positive Hopf band is obtained.  Indeed, beginning from Figure~\ref{fig:posfiberedmonty} (left), apply Lemma~\ref{lem:obvioushopf} to obtain a sequence of deplumbings that transform each $d_i$ down to $1$.  This results in the odd type fibered Montesinos link $M(-\tfrac{1}{3}, -\tfrac{1}{3}, \dots, -\tfrac{1}{3} \vert 1)$ of length $r$.  As demonstrated in Figure~\ref{fig:oddpretzel}, each $-\tfrac{1}{3}$ signifies a positive Hopf band that may be deplumbed, deleting that term from the notation.  After $r-1$ steps, we are left with the link $M(-\tfrac{1}{3} \vert 1)$ which is itself a positive Hopf band. Reversing this process exhibits the fiber of our original knot as a successive Hopf plumbing.
\end{proof}

\begin{figure}[h]
	\centering
                	\includegraphics[scale=.3]{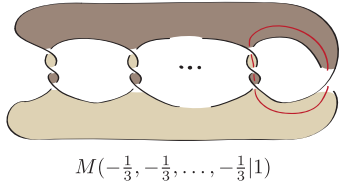}
	\caption{$M(-\tfrac{1}{3}, -\tfrac{1}{3}, \dots, -\tfrac{1}{3} \vert 1)$ has a positive Hopf band that may be deplumbed.}
    	\label{fig:oddpretzel}
\end{figure}

\begin{figure}[h]
	\centering
                	\includegraphics[scale=.3]{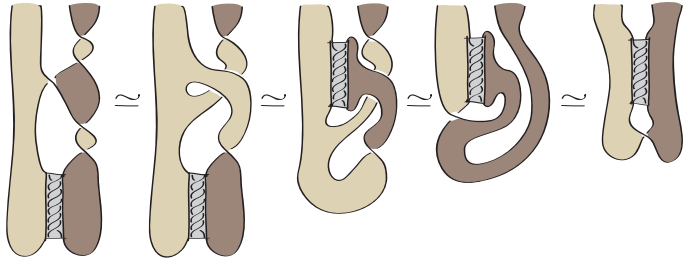}
	\caption{An isotopy illustrating the equivalence $[\dots, a-2,-1,-2,d] = [\dots, a, d+1]$}
    	\label{fig:212b_to_b1}
\end{figure} 

\begin{lemma}
\label{lem:detodd}
Let $K$ be an odd fibered Montesinos knot supporting the tight contact structure. Then $\det(K) > 2g(K) + 1$ unless
 $K=  M(\tfrac{1}{3}, \tfrac{1}{3}, \tfrac{2}{5} \vert1)$.
 \end{lemma}
\begin{proof}
By Theorem 3.1 \cite{HM-monty}, an odd fibered Montesinos knot $K$ has genus
\[
	g(K) = \frac{1}{2} \left( \sum_{i=1}^r b^{(i)} + |e| -1 \right)
\]
where 
\[
	b^{(i)} = \sum_{j=1}^{q_i} b_j^{(i)}.
\]
By Proposition \ref{prop:correctform}, the parameters of any such $K$ supporting the tight contact structure are $r\geq3$, $e=+1$ and
$\frac{\beta_i}{\alpha_i} = \frac{-d_i}{2d_i+1} = [-2,d_i]$ for some integers $d_i>0$, for all $i=1, \dots, r$ (i.e. $q_i = 1$ for all $i$). Thus,
\[
	2g(K)+1 =\sum_{i=1}^r d_i + 1.
\]
Using Lemma~\ref{lem:determinant}, we have
\[\det(K) =  2^r\prod_{i=1}^r (d_i+\frac{1}{2}) \left|  1 - \sum_{i=1}^r \frac{d_i}{2d_i+1} \right|.\]

Since $d_i>0$ for all $i$, then
\[ 
\tag{$\dagger$}	\frac{r}{2} > \sum_{i=1}^r \frac{d_i}{2d_i+1} \geq \frac{r}{3}.
\]
When $r>3$, $\frac{r}{3}>1$ and applications of ($\dagger$) then gives
\[ 2^r \left|  1 - \sum_{i=1}^r \frac{d_i}{2d_i+1} \right| = 2^r \left(  \sum_{i=1}^r \frac{d_i}{2d_i+1} -1 \right)\geq 2^r \left( \frac{r}{3} - 1 \right) > 1.\]
Hence 
\[
\det(K) > \prod_{i=1}^r (d_i+\frac{1}{2}) > \sum_{i=1}^r d_i + 1 = 2g(K)+1.
\]

When $r=3$, then
\[
	\det(K) = 8\prod_{i=1}^3 (d_i+\frac{1}{2}) \left| 1- \sum_{i=1}^3 \frac{d_i}{2d_i+1} \right|.
\]
If $1 \geq \sum_{i=1}^{3} \frac{d_i}{2d_i+1}$, then necessarily $d_1=d_2=d_3=1$ and $\det(K) = 0$.  However $M(\frac{1}{3},\frac{1}{3},\frac{1}{3}\vert1)$ is a two component link rather than a knot.

Otherwise, at least one of $d_1, d_2, d_3$ is greater than $1$.  Thus
$ \left| 1- \sum_{i=1}^3 \frac{d_i}{2d_i+1} \right| \geq \frac{1}{15}$ and $\det(K)  \geq \frac{8}{15}\prod_{i=1}^3 (d_i+\frac{1}{2})$ with equality only when $\{d_1, d_2, d_3\} = \{1,1,2\}$.

If $d_1, d_2, d_3 \geq 2$ or if $d_1=1$ and $d_2, d_3\geq 3$, then
\[
	\det(K) > \frac{8}{15}\prod_{i=1}^3 (d_i+\frac{1}{2}) > \sum_{i=1}^3 d_i +1 =2g(K)+1.
\]
Otherwise, the triple $\{d_1, d_2, d_3\}$ is one of  $\{1, 1, 2\}$, $\{1, 1, 3\}$, $\{1, 2, 2\}$ or $\{1, 2, 3\}$. The triples  $\{1, 1, 3\}$ and $\{1,2, 2\}$ correspond with two component links rather than knots.  An explicit computation for the remaining two triples finishes the proof.  For the triple $\{1, 2, 3\}$, we have $\det(K) = 17$ while $2g(K)+1 = 7$.   For the triple $\{1, 2, 3\}$, we have $\det(K) = 3$ while $2g(K)+1 = 5$.  This last one gives the knot stated in the theorem.   
\end{proof}

\begin{proposition}\label{prop:nooddlspace}
No odd type  Montesinos knot (of length $r\geq3$) is an L-space knot.
\end{proposition}

\begin{proof}
Since L-space knots are fibered knots that support the tight contact structure and satisfy the determinant-genus bound, Lemma~\ref{lem:detodd} identifies $M(\tfrac{1}{3}, \tfrac{1}{3}, \tfrac{2}{5} \vert1)$ as the single candidate for an L-space knot among the odd type Montesinos knots (since we assumed throughout that $r\geq3$).  This knot may be identified as the knot $10_{145}$ in the Rolfsen table.  Its  Alexander polynomial is $\Delta_{10_{145}(t)} = t^2+t-3+t^{-1}+t^{-2}$. This fails the Alexander polynomial condition for L-space knots.
\end{proof}

%%%%%%%%%%%%%%%%%%%%%%%%%%%%%%%%%%%
%%%%%%%%%%%%%%%%%%%  EVEN  %%%%%%%%%%%%
%%%%%%%%%%%%%%%%%%%%%%%%%%%%%%%%%%%

\subsection{Even fibered Montesinos knots}

\begin{proposition}
\label{prop:eventight}
Let $K$ be an even fibered Montesinos knot supporting the tight contact structure. Then for some set of positive integers $m_1, \dots, m_r$ with $r\geq 3$,
\[
	K = M( \tfrac{-m_1}{-m_1+1},\dots, \tfrac{-m_r}{-m_r+1} \vert 2 ).
\]
Moreover, the fiber of $K$, shown on the right side of Figure \ref{fig:posfiberedmonty}, is a positive Hopf plumbing.
\end{proposition}

\begin{proof}
Let $K$ be a fibered Montesinos knot of even type. 
By Theorem 3.2 of Hirasawa-Murasugi \cite{HM-monty}, we may assume that $e$ is even and each $\beta_i/\alpha_i$ has an even continued fraction expansion
\[
	S_i = [ 2c_1^{(i)}, 2c_2^{(i)}, 2c_3^{(i)}, \dots, 2c_{m_i}^{(i)} ].
\]
 The fiber $F$ of $K$ appears as in Figure 12 or Figure 13 of \cite{HM-monty} depending on whether $e\neq0$ or $e=0$ respectively.
(Since $K$ is a knot, we may further assume that $m_1$ is odd and $m_i$ is even for $i>1$, but we will not use this.  The following arguments apply equally well to any fibered Montesinos link with fiber $F$ as in Figure 12 or Figure 13 of \cite{HM-monty}; indeed the proof of the conditions for fiberedness in Theorem 3.2 of \cite{HM-monty} do not rely upon the connectedness of $K$.)

Assume $e=0$. 
Then since we assume $r\geq3$,  conditions E(II-i) and E(III) guarantee there exists an index $i$ such that $(c_1^{(i)},c_1^{(i+1)}) = \pm(1,-1)$.   Then, referring to Figure 13 of \cite{HM-monty} for the surface $F$, there is an unknot $C$ in $F$ running once through each of the bands corresponding to $c_1^{(i)}$ and $c_1^{(i+1)}$ such that $C$ is $0$--framed by $F$.  In particular $C$ is a twisting loop in $F$. By Lemma \ref{lem:OTmarkers} any such knot supports an overtwisted contact structure, and thus we conclude $e \neq 0$.

Since $e\neq0$ and $K$ is fibered, E(I) implies that $e=\pm2$. Referring to Figure 12 of \cite{HM-monty} for the surface $F$, it is readily apparent that $F$ contains a negative Hopf band if $e=-2$ and a positive one if $e=2$.  Due to Lemma \ref{lem:OTmarkers}, we must have $e=2$.

Condition E(I) with Lemma \ref{lem:OTmarkers} (and Lemma~\ref{lem:twobridge} and its proof) then further imply
 $c_j^{(i)}=-1$ for all $i,j$ so that for each $i$ we have the continued fraction $S_i = [-2,-2, \dots, -2]$ of length $m_i$.  This in turn implies $\frac{\beta_i}{\alpha_i}=\frac{-m_i}{m_i+1}$. Thus we have the fibered Montesinos knot  
 \[
	K = M( \tfrac{-m_1}{m_1+1},\dots, \tfrac{-m_r}{m_r+1} \vert 2 )
\]
for some set of positive integers $m_1, \dots, m_r$, which appears as in the right side of Figure \ref{fig:posfiberedmonty}. Furthermore the fiber $F$ may be assembled by beginning with the positive Hopf band corresponding to the $e=2$ twists and then plumbing on to it a linear chain of $m_i$ positive Hopf bands for each $i$.  Hence the fibered knot is a positive Hopf plumbing and supports the tight contact structure.
\end{proof}

\begin{lemma}\label{lem:montypretz}
We have the following equivalence between Montesinos links and pretzel links.
For any positive integers $m_1, \dots, m_r$, 
\[M( \tfrac{-m_1}{m_1+1},\dots, \tfrac{-m_r}{m_r+1} \vert 2 )= P(m_1+1, \dots, m_r+1,\underbrace{-1,\dots , -1}_{r-2})\]
\end{lemma}
\begin{proof}
Since $\tfrac{-m_i}{m_i+1} = -1+\tfrac{1}{m_i+1}$, the Montesinos link $M( \tfrac{-m_1}{m_1+1},\dots, \tfrac{-m_r}{m_r+1} \vert 2 )$ is the pretzel link $P(-1, m_1+1, \dots, -1, m_r+1,1, 1)$ as shown in the right side of Figure~\ref{fig:posfiberedmonty}. By flype moves, $P(-1, m_1+1, \dots, -1, m_r+1,1, 1)$ is isotopic to the pretzel link $P(m_1+1, \dots, m_r+1,\underbrace{-1,\dots , -1}_{r-2})$.
\end{proof}

\begin{proposition}\label{prop:evenLspace}
The only even type Montesinos knots (of length $r \geq 3$) that are L-space knots are the pretzel knots $P(-2, 3, 2n+1)$ for integers $n\geq0$.
\end{proposition}

\begin{proof}
Since L-space knots must be fibered knots that support the tight contact structure, Proposition~\ref{prop:eventight} restricts the candidates for L-space knots among the even type Montesinos knots to those of the form $M( \tfrac{-m_1}{-m_1+1},\dots, \tfrac{-m_r}{-m_r+1} \vert 2 )$ for positive integers $m_1, \dots, m_r$ with $r \geq 3$.  Lemma~\ref{lem:montypretz} shows that these Montesinos knots are actually pretzel knots.  According to \cite{lidmanmoore}, the only L-space pretzel knots (that are not two-bridge knots) are the pretzel knots $P(-2,3,2n+1)$ for integers $n\geq0$.  Noting that $P(-2,3,2n+1)$ and $P(2,3,2n+1,-1)$ are isotopic pretzel knots completes the proof.
\end{proof}

\begin{remark}
One could obtain Proposition~\ref{prop:evenLspace} without appealing to \cite{lidmanmoore}.  Proceeding with the general procedure as we did for the odd type Montesinos knots, one may cull the knots from Proposition~\ref{prop:eventight} with the determinant-genus relation.  This shows that if $K$ is an even fibered Montesinos knot supporting the tight contact structure such that $\det(K) \leq 2g(K)+1$, then $K= M \left( \tfrac{-m_1}{m_1+1}, \tfrac{-m_2}{m_2+1}, \tfrac{-m_3}{m_3+1} \vert 2 \right)$ where $(m_1,m_2,m_3)  \in \{(5,2,2),(3,2,2),(3,2,4),$ $(1,4,4),(1,4,6),(1,2,2c)\}$ for positive integers $c$.   The Alexander polynomials of these first five of these knots have coefficients that are greater than one in absolute value; hence those knots cannot be L-space knots.  The remaining family of knots is the desired family of pretzel knots $P(-2,3,2n+1)$ for non-negative integers $n$.
\end{remark}

%%%%%%%%%%%%%%%%%%%%%%%
%%%%%%%  ESSENTIAL TANGLE DECOMPOSITIONS  %%%%%%%%%%%
%%%%%%%%%%%%%%%%%%%%%%%

\section{On essential tangle decompositions of L-space knots}\label{sec:esstangles}
A knot $K$ in  $3$--manifold  $Y$ has an {\em essential $n$--string tangle decomposition} if there is an embedded sphere $Q$ that transversally intersects $K$ in $2n$ points such that the planar surface $Q -\bdry \nbhd(K)$ is essential in the knot exterior $Y-\nbhd(K)$, i.e.\ $Q -\bdry \nbhd(K)$  is both incompressible and $\bdry$--incompressible.  If $K$ has no essential $n$--string tangle decomposition, then $K$ is called {\em $n$--string prime}. 

Observe that a $1$--string prime knot is simply a prime knot.  
\begin{theorem}[Krcatovich \cite{krcatovich}]
 L-space knots are prime knots.  
\end{theorem}

A {\em Conway sphere} for a knot (or link) is a sphere $Q$ transversally intersecting the knot in $4$ points.
Hence a $2$--string prime knot is a knot without any essential Conway spheres.   Thus the first conjecture from the introduction may be rephrased as follows:

\begin{conjecture}[Lidman-Moore \cite{lidmanmoore}]\label{conj:essconway}
L-space knots are $2$-string prime.
\end{conjecture}

Since Montesinos knots generalize pretzel knots, and those with length $r\geq 4$ have essential $2$--string tangle decompositions (for example, a sphere separating two adjacent factors from the remaining $r-2$ is an essential Conway sphere),  our Theorem~\ref{thm:main} lends further credence to this conjecture. However we suspect something stronger is true.

\begin{conjecture}\label{conj:nstringprime}
L-space knots are $n$--string prime for all integers $n>0$.  That is, 
L-space knots have no essential tangle decomposition.
\end{conjecture}

Since an essential tangle decomposition of a knot can give rise to a closed essential surface in the knot's exterior, one may be tempted to conjecture that L-space knots have no closed essential surfaces in their exterior.  Let us note however, that there are hyperbolic L-space knots  for which this is not the case (indeed, ones with lens space surgeries, e.g.\ \cite{baker-closedessentialsurfaces}).  For non-hyperbolic examples, there are satellite L-space knots.

\subsection{Satellite operations}
Recall that for a satellite knot $K$, there is a knotted solid torus $V$ containing $K$  such that there is no isotopy of $K$ in $V$ to the core of $V$.  The core of $V$ is the companion of $K$, and the pair $(V,K)$ is the pattern of $K$.   If $K$ is braided in $V$, then we say $K$ is a {\em braided satellite} of the core of $V$. 

Hayahsi-Matsuda-Ozawa show that if there is no essential tangle decomposition of the pattern $(V,K)$, then any essential tangle decomposition of the satellite $K$ gives rise to an essential tangle decomposition of the core of $V$, \cite{hmo}.  In particular, this gives the following result for braided satellites.
\begin{theorem}[Theorem 4.1\cite{hmo}]\label{thm:braidsat}
A braided satellite of a knot with no essential tangle decomposition also has no essential tangle decomposition.
\end{theorem}

The Berge-Gabai knots are the knots in solid tori with non-trivial Dehn surgeries yielding solid tori \cite{bergeS1xD2,gabaiS1xD2}.  In particular, a satellite knot $K$ with one of these as its pattern has a non-trivial Dehn surgery that is equivalent to a Dehn surgery on its companion knot.  Thus, when done with the correct framings, the operation of taking Berge-Gabai satellites (which includes cabling) preserves the property of being L-space knots.  But a sharper statement can be made.

\begin{theorem}[Hom-Lidman-Vafaee \cite{HLV}]\label{thm:BGsat}
An L-space knot that is a Berge-Gabai satellite must have an L-space knot as its companion.
\end{theorem}
Previously, Hedden \cite{hedden-cable} and Hom \cite{hom} determined how the property of being an L-space knot behaves with respect to cabling.  Together Theorem~\ref{thm:braidsat} and Theorem~\ref{thm:BGsat} show that if a Berge-Gabai satellite knot is a counterexample to Conjecture~\ref{conj:nstringprime} then so is its companion.

\begin{question}
Let $K$ be an L-space satellite knot.
\begin{itemize}
\item Is $K$ a braided satellite?
\item Is the companion knot to $K$ also an L-space knot?
\end{itemize}
\end{question}

\subsection{L-space knots of large tunnel number}
Recall that the {\em tunnel number} of a knot $K$  is the minimum number of mutually disjoint,  embedded arcs intersecting  $K$ at their endpoints such that the exterior of the resulting $1$--complex is a handlebody.

The L-space knots identified in Theorem~\ref{thm:main} and all of the Berge knots \cite{berge} have tunnel number one, as do many other familiar L-space knots.  
 Gordon-Reid have shown that tunnel number one knots are $n$--string prime for all integers $n>0$ \cite{gordonreid}.   Thus  L-space knots with tunnel number one support Conjecture~\ref{conj:nstringprime}.

However there are L-space knots with greater tunnel number.  
Indeed, sufficiently large cables of  L-space knots are also  L-space knots \cite{hedden-cable}, yet tunnel number one cabled knots are only those which are certain cables of torus knots \cite{morimotosakuma-TN1satellite}.  More generally we show there are L-space knots of arbitrarily large tunnel number.

\begin{proposition}\label{prop:tn}
For any integer $N$, there is an L-space knot with tunnel number greater than $N$.
\end{proposition}

\begin{proof}
Sufficiently large cables of positive L-space knots are also positive L-space knots \cite{hedden-cable}. Thus it is enough to show that, generically, the tunnel number of an iterated cable of a torus knot (i.e.\ an iterated torus knot) grows with the number of cabling iterations.   Indeed, this follows from Theorem~4.2 of \cite{zupan-bridgespectra}. 
\end{proof}

The L-space knots with tunnel number greater than one constructed in Proposition~\ref{prop:tn} are all satellite knots.  Furthermore, we have not identified a non-satellite L-space knot in the literature without tunnel number one, though we expect there should be many.  Is this the case?
\begin{question}
Is there a non-satellite L-space knot with tunnel number greater than one?
\end{question}

Along these lines, classifying L-space knots among tunnel number one knots would be informative.   
\begin{question}
Which tunnel number one knots are L-space knots?
\end{question}
Note that our general strategy in section~\ref{sec:strategy} barely gets off the ground since fiberedness among tunnel number one knots is not yet well understood.
\begin{question} \quad
\begin{itemize}
\item[A.] Which tunnel number one knots in $S^3$ are fibered?  \cite{johnson-ldtblog}
\item[B.]
Which fibered tunnel number one knots in $S^3$ support the tight contact structure?
\end{itemize}
\end{question}

\section{Acknowledgements}
We would like to thank the organizers of the Low Dimensional Topology workshop of 2013 held at the Simons Center for Geometry and Physics where this work was initiated. The first author was partially supported by Simons Foundation grant 209184 and the second author was partially supported by NSF grant DMS-1148609.

\begin{footnotesize}
	\bibliographystyle{alpha}
	\bibliography{montybibly}
\end{footnotesize}

\end{document}